\documentclass[11pt]{jacodesmath-revised}

\usepackage{verbatimbox}
\usepackage[capitalize]{cleveref}
\usepackage{doi}

\newcommand{\pref}[1]{(\ref{#1})}
\newcommand{\fullref}[2]{\ref{#1}\pref{#1-#2}}
\newcommand{\fullcref}[2]{\cref{#1}\pref{#1-#2}}
\newcommand{\fullCref}[2]{\Cref{#1}\pref{#1-#2}}

\newcommand{\cartprod}{\mathbin{\Box}}
\newcommand{\crt}[2]{a_{m,n}(#1, #2)}
\newcommand{\dC}{\vec C}
\DeclareMathOperator{\lcm}{lcm}
\renewcommand{\pmod}[1]{\ (\mathop{\mathrm{mod}} #1)}
\newcommand{\ZZ}{\mathbb{Z}}

\numberwithin{equation}{section}

\newtheorem{cor}[equation]{Corollary}
\newtheorem{lem}[equation]{Lemma}
\newtheorem{prop}[equation]{Proposition}
\newtheorem{thm}[equation]{Theorem}

\crefname{cor}{corollary}{corollaries}
\Crefname{cor}{Corollary}{Corollaries}
\crefformat{cor}{Corollary~#2#1#3}
\crefname{lem}{lemma}{lemmas}
\Crefname{lem}{Lemma}{Lemmas}
\crefformat{lem}{Lemma~#2#1#3}
\crefname{prop}{proposition}{propositions}
\Crefname{prop}{Proposition}{Propositions}
\crefformat{prop}{Proposition~#2#1#3}

\theoremstyle{definition}

\newtheorem{defn}[equation]{Definition}

\newtheorem{notn}[equation]{Notation}

\newtheorem{rems}[equation]{Remarks}

\crefname{defn}{definition}{definitions}
\Crefname{defn}{Definition}{Definitions}
\crefformat{defn}{Definition~#2#1#3}

\crefformat{figure}{Figure~#2#1#3}
\crefformat{rems}{Remark~#2#1#3}

\newcounter{subfigure}
\renewcommand{\thesubfigure}{\alph{subfigure}}

\renewenvironment{abstract}{%
\hfill\begin{minipage}{0.95\textwidth}
\rule{\textwidth}{1pt} 
\vskip1.5\smallskipamount\par % this is what I added !!!
\textbf{Abstract.}}
{\par\noindent\rule{\textwidth}{1pt}\end{minipage}}
\newcommand{\Zbl}[1]{\href{https://zbmath.org/?q=an:#1}{Zbl\,#1}}
\newcommand{\MR}[1]{\href{https://mathscinet.ams.org/mathscinet-getitem?mr=#1}{MR\,#1}}

\begin{document}

\title{ \textbf{Hamiltonicity after reversing the directed edges \\ at a vertex of a Cartesian product} }

\author{Dave Witte Morris\thanks{dmorris@deductivepress.ca, https://deductivepress.ca/dmorris}}
\affil{Department of Mathematics and Computer Science, 
\\ University of Lethbridge, 
\\ Lethbridge, Alberta, T1K~6R4, Canada}

\maketitle

\begin{abstract}
Let $\dC_m$ and~$\dC_n$ be directed cycles of length $m$ and~$n$, with $m,n \ge 3$, and let $P(\dC_m \cartprod \dC_n)$ be the digraph that is obtained from the Cartesian product $\dC_m \cartprod \dC_n$ by choosing a vertex~$v$, and reversing the orientation of all four directed edges that are incident with~$v$. (This operation is called ``pushing'' at the vertex~$v$.) By applying a special case of unpublished work of S.\,X.\,Wu, we find elementary number-theoretic necessary and sufficient conditions for the existence of a hamiltonian cycle in $P(\dC_m \cartprod \dC_n)$. 

\qquad % not for publication version !!!
A consequence is that if $P(\dC_m \cartprod \dC_n)$ is hamiltonian, then $\gcd(m,n) = 1$, which implies that $\dC_m \cartprod \dC_n$ is not hamiltonian. This final conclusion verifies a conjecture of J.\,B.\,Klerlein and E.\,C.\,Carr.

\smallskip % not for publication version !!!

\\[\smallskipamount] % \smallskipamount not for publication version !!!
\textbf{2020 AMS Classification.} 
	05C45, % Eulerian and Hamiltonian graphs
	05C20, % Directed graphs (digraphs), tournaments
	05C76. % Graph operations (line graphs, products, etc.)
\end{abstract}

\setcounter{section}{-1}

\section{Preliminaries}

\begin{notn}
For $m,n,i,j \in \ZZ$ (with $m \neq 0$ and $n \neq 0$), we define the integer $\crt{i}{j}$ by the following conditions: 
	\[ \text{$\crt{i}{j} \equiv i \pmod{m}$,
	\ 
	$\crt{i}{j} \equiv j \pmod{n}$,
	\ and \ 
	$1 \le \crt{i}{j} \le \lcm(m, n)$}
	. \]
The integer is unique, if it exists. By the Chinese Remainder Theorem, $\crt{i}{j}$ does exist whenever $\gcd(m,n)= 1$ (or, more generally, whenever $i \equiv j \pmod{\gcd(m,n)}$).
\end{notn}

\begin{notn}
We use $\dC_m$ to denote a directed cycle of length~$m$.
\end{notn}

\begin{defn}[\hskip-.0001pt% not for publication version !!!
{\cite[p.~88]{KlerleinCarr}}]
If $X$ is a digraph that is vertex-transitive, then the digraph $P(X)$ is constructed from~$X$ by choosing a vertex~$v$, and reversing the orientation of each directed edge that is incident with~$v$. (This operation is called ``pushing'' at the vertex~$v$ \cite{KlerleinCarr, Klostermeyer}.) Since $X$ is vertex-transitive, the isomorphism class of the resulting digraph is independent of the choice of~$v$.
\end{defn}

\begin{defn}[\hskip-.0001pt% not for publication version !!!
{\cite[pp.~35 and~421]{ProductHandbook}}]
Recall that the \emph{Cartesian product} $X \cartprod Y$ of two digraphs $X$ and~$Y$ is the digraph whose vertex set is $V(X) \times V(Y)$, with a directed edge from $(x_1,y_1)$ to $(x_2, y_2)$ if and only if either 
	\begin{itemize}
	\item $x_1 = x_2$, and there is a directed edge from $y_1$ to~$y_2$ in~$Y$,
	or
	\item $y_1 = y_2$, and there is a directed edge from $x_1$ to~$x_2$ in~$X$.
	\end{itemize}
\end{defn}

\section{Statement of the main result}

This note explains that a special case of unpublished work of S.\,X.\,Wu \cite{Wu} (or slightly later published work of S.\,J.\,Curran et al.\ \cite{CurranEtAl}) provides the following elementary number-theoretic necessary and sufficient conditions for $P(\dC_m \cartprod \dC_n)$ to be hamiltonian.

\begin{prop} \label{main}
Let $\dC_m$ and $\dC_n$ be directed cycles of length $\ge 3$. The digraph $P(\dC_m \cartprod \dC_n)$ has a hamiltonian cycle if and only if
	\begin{enumerate}
	\item \label{main-mn}
	$\gcd(m, n) = 1$,
	\item \label{main-20}
	$\min\{\, \crt{0}{-2}, \ \crt{-2}{0} \,\} < \min \bigl\{\, \crt{0}{-1}, \, \crt{-1}{0} \, \bigr\}$,
	and
	\item \label{main-40}
	 $\displaystyle \gcd \left( \frac{\crt0{-4}}{m} , \ \frac{\crt{-4}{0}}{n} \right) = 1$.
	\end{enumerate}
\end{prop}

%\begin{prop} \label{main}
%Let $\dC_m$ and $\dC_n$ be directed cycles of length $\ge 3$. The digraph $P(\dC_m \cartprod \dC_n)$ has a hamiltonian cycle if and only if\/ $\gcd(m, n) = 1$ and, perhaps after interchanging $m$ and~$n$, we have
%	\[ \text{$\crt20 > \max \bigl\{\, \crt02, \, \crt12, \, \crt21, \, \crt22 \, \bigr\}$, 
%	\quad
%	$n$ is odd,}
%	\]
%and
%	\[ \gcd \left( \frac{\crt0{-4}}{m} , \ \frac{\crt{-2}{0}}{n} \right) = 1 . \]
%\end{prop}

If $\gcd(m,n) = 1$, then it is well known (and easy to see) that $\dC_m \cartprod \dC_n$ is not hamiltonian \cite[Thm.~28.1, p.~510]{GallianBook}. Therefore, the \lcnamecref{main} has the following consequence, which was conjectured by J.\,B.\,Klerlein and E.\,C.\,Carr \cite[p.~94]{KlerleinCarr}:

\begin{cor} \label{HamNotHam}
If $\dC_m \cartprod \dC_n$ is hamiltonian \textup(and $m,n \ge 3$\textup), then $P(\dC_m \cartprod \dC_n)$ is not hamiltonian.
\end{cor}

\begin{rems} \label{mainRems}
\leavevmode
	\begin{enumerate}
	\item \Cref{main} requires $m$ and~$n$ to be at least~$3$. The remaining case was settled by J.\,B.\,Klerlein and E.\,C.\,Carr \cite[Thm.~6]{KlerleinCarr}:
	 $P(\dC_2 \cartprod \dC_n)$ is hamiltonian if and only if $n \in \{2,3\}$.
	 (Since $\dC_2 \cartprod \dC_2$ and $P(\dC_2 \cartprod \dC_2)$ are hamiltonian, it is clear that \cref{HamNotHam} would be false if it allowed the case where $m = n = 2$.)
	\item \label{mainRems-3}
	J.\,B.\,Klerlein and E.\,C.\,Carr also determined whether $P(C_m \cartprod C_n)$ is hamiltonian in certain other special cases. In particular, they \cite[Thm.~7]{KlerleinCarr} proved a much more concrete form of the case $m = 3$ of \cref{main}: $P(C_3 \cartprod C_n)$ is hamiltonian if and only if $n \equiv 2 \pmod{3}$.
	\item The conditions in \cref{main} are so efficient that they can be checked by a computer in seconds, even if $m$ and~$n$ have 100,000 digits. This can be verified by using the sample code in \cref{code}.
	\end{enumerate}
\end{rems}

\begin{verbbox}
def is_PCmxCn_hamiltonian(m, n):
    r"""Return `True` if `P(C_m x C_n)` has a hamiltonian cycle
    (otherwise return `False`)."""
    m = Integer(m)
    n = Integer(n)
    if min(m, n) < 3:
        raise NotImplementedError("m and n must be at least 3")
    if gcd(m, n) != 1:
        return False
    def a(i, j):
        return crt(i, j, m, n)
    if min( a(0, -2), a(-2, 0) ) > min( a(0, -1), a(-1, 0) ):
        return False
    return gcd( a(0, -4) // m, a(-4, 0) // n ) == 1
\end{verbbox}
\begin{figure}[ht]
\centerline{\theverbbox}
\caption{A \textsf{sagemath} program that implements \cref{main}.
(This program can be run online at \url{https://cocalc.com}.)
For example, \texttt{is\_PCmxCn\_hamiltonian(3,5)} returns \texttt{True} because the digraph $P(\dC_3 \cartprod \dC_5)$ is hamiltonian (see \fullcref{mainRems}{3}).}
\label{code}
\end{figure}

\section{Proof of the main result}

We assume that the vertices of $\dC_m \cartprod \dC_n$ are identified in the natural way with the elements of the abelian group $\ZZ_m \times \ZZ_n$.

\begin{notn}[cf.\ {\cite[p.~2]{Wu}}]
\leavevmode
	\begin{enumerate}
	\item For $a,b \in \ZZ^+$, the rectangle~$R_{a,b}$ is the subset $\{0,1,\ldots, a-1\} \times \{0,1,\ldots,b-1\}$ of $V(\dC_m \cartprod \dC_n)$.
	\item We use $(\dC_m \cartprod \dC_n) \smallsetminus R_{a,b}$ to denote the digraph that is obtained from $\dC_m \cartprod \dC_n$ by deleting all of the vertices in $R_{a,b}$ (and also deleting all of the directed edges that are incident with this set).
	\end{enumerate}
\end{notn}

The following simple observation is crucial: 

\begin{lem} \label{PushedIffDeleted}
For $m,n \ge 3$, the digraph $P(\dC_m \cartprod \dC_n)$ is hamiltonian if and only if $(\dC_m \cartprod \dC_n) \smallsetminus R_{2,2}$ is hamiltonian.
\end{lem}

\begin{proof}
($\Rightarrow$) \fullCref{PushedFig}{AssumeP} shows a part of $P(\dC_m \cartprod \dC_n)$ with the pushed vertex~$v$ at its centre. (All nine vertices in the figure are distinct, because $m,n \ge 3$.) Note that the vertices $v + (1,0)$ and $v + (0,1)$ have only one in-edge, and the vertices $v - (1,0)$ and $v - (0,1)$ have only one out-edge. This implies that the hamiltonian cycle must traverse these four directed edges (which are dark in the figure). 

The out-edges of~$v$ go to $v - (1,0)$ and $v - (0,1)$.
By symmetry (i.e., by interchanging $m$ and~$n$ if necessary), we may assume without loss of generality that the hamiltonian cycle uses the (white) directed edge from~$v$ to $v - (1,0)$. Then the hamiltonian cycle cannot use the edge from $v + (0,1)$ to~$v$ (because that would create a $4$-cycle), so it must use the other in-edge of~$v$, which is the (grey) directed edge from $v + (1,0)$ to~$v$. Also, the hamiltonian cycle cannot use the (striped) directed edge from $v - (1,1)$ to $v - (1,0)$ (because it already uses a different in-edge of $v - (1,0)$), so it must use the other out-edge of $v - (1,1)$, which goes to $v - (0,1)$ (and is grey in the figure). 

Now, assuming without loss of generality that $R_{2,2}$ consists of the four white vertices in the bottom right of the picture, we can construct a hamiltonian cycle in $(\dC_m \cartprod \dC_n) \smallsetminus R_{2,2}$ by deleting the edges in the walk 
	\[ v - (1,1), \ v - (0,1), \ v + (1,-1), \ v + (1,0), \ v, \ v - (1,0) , \]
and inserting the (striped) directed edge from $v - (1,1)$ to $v - (1,0)$.

\begin{figure}[ht]
\setcounter{subfigure}{0}
\centerline{\refstepcounter{subfigure} \label{PushedFig-AssumeP}
	$\begin{matrix} \includegraphics{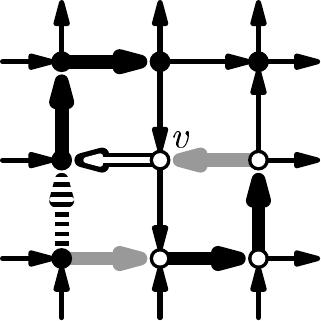} \\ \text{(\thesubfigure)} \end{matrix}$
	\hskip 0pt plus 0.5fil
	\refstepcounter{subfigure} \label{PushedFig-AssumeR}
	$\begin{matrix} \includegraphics{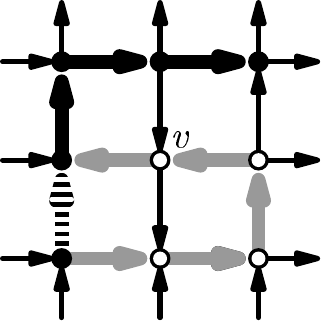} \\ \text{(\thesubfigure)} \end{matrix}$}
	\caption{Two drawings centred at the pushed vertex~$v$.}
	\label{PushedFig}
\end{figure}

\medbreak

($\Leftarrow$) \fullCref{PushedFig}{AssumeR} shows the same portion of $P(\dC_m \cartprod \dC_n)$, centred at the pushed vertex~$v$, with the (white) vertices of the rectangle $R_{2,2}$ in the bottom right corner again. Note that 
	\[ P(\dC_m \cartprod \dC_n) \smallsetminus R_{2,2} = (\dC_m \cartprod \dC_n) \smallsetminus R_{2,2}, \]
so, by assumption, there is a hamiltonian cycle in $P(\dC_m \cartprod \dC_n) \smallsetminus R_{2,2}$.  It must use all of the directed edges that are dark or striped in this picture, because $v - (1,1)$ and $v - (1,0)$ have only one out-edge that has not been deleted, and the vertices $v + (0,1)$ and $v + (1,1)$ have only one in-edge that has not been deleted. Then we can construct a hamiltonian cycle in $P(\dC_m \cartprod \dC_n)$ by reversing the process in the previous part of the proof: delete the (striped) edge from $v - (1,1)$ to $v - (1,0)$, and replace it with the walk
	\[ v - (1,1), v - (0,1), v + (1,-1), v + (1,0), v, v - (1,0) , \]
whose edges are grey in the picture.
\end{proof}

Hence, the existence of a hamiltonian cycle in $P(\dC_m \cartprod \dC_n)$ is characterized by the case $a = b = 2$ of the following result:

\begin{thm}[S.\,X.\,Wu {\cite[Cor.~11]{Wu}}] \label{WuThm}
The digraph $(\dC_m \cartprod \dC_n) \smallsetminus R_{a,b}$ is hamiltonian if and only if either the following conditions are satisfied, or they are satisfied after interchanging $m$ and~$n$, and also interchanging $a$ and~$b$: 
	\[ \text{$\crt{-a}{0}$ exists,} \]
	\[ \crt{-a}{0} = \min \left\{ \begin{matrix}
		 \crt{-a}{-b}, \crt{-a}{-b + 1}, \crt{-a}{-b + 2}, \ldots, \crt{-a}{0}, \\[\smallskipamount]
		 \crt{-a}{-b}, \crt{-a + 1}{-b}, \crt{-a + 2}{-b}, \ldots, \crt0{-b} 
		 \end{matrix} \right\} \]
\textup(where any terms in the minimum that do not exist are simply ignored\/\textup), and
	\[ \gcd \left( n - b - b \left\lfloor \frac{\crt{-a}{0}}{m} \right\rfloor, \ b \, \frac{\crt{-a}{0}}{n} \right) = 1 . \]
\end{thm}

\begin{rems}
\leavevmode
	\begin{enumerate}
	\item S.\,X.\,Wu showed that if $(\dC_m \cartprod \dC_n) \smallsetminus R_{a,b}$ has a hamiltonian cycle, then it is unique (see \cref{UniqueCycles} below). % !!!
 It follows that if $P(\dC_m \cartprod \dC_n)$ is hamiltonian (and $m,n \ge 3$), then $P(\dC_m \cartprod \dC_n)$ has exactly two hamiltonian cycles. One hamiltonian cycle will be constructed in the proof of \cref{PushedIffDeleted}, and the other is constructed by interchanging $\dC_m$ and~$\dC_n$ in this proof (or, in other words, by reflecting \fullcref{PushedFig}{AssumeR} across the line $y = x$).

%	\item Since \cite{Wu} is unpublished, \cref{WuProofSect} will sketch a proof of the case that we need.
	\item A very different formulation of the conditions in the statement of \cref{WuThm} was proved by S.\,J.\,Curran et al.\ \cite[Thm.~4.3]{CurranEtAl}, as a special case of a more general version \cite[Thm.~4.2]{CurranEtAl} that applies to all $2$-generated Cayley digraphs on finite abelian groups, not only Cartesian products of directed cycles.
	\end{enumerate}
\end{rems}

We need only the special case where $a = b = 2$, which can be restated as follows (see \cref{DerivePfSect} or \cref{DirectPfSect}):

\begin{cor} \label{WuThm2}
For $m,n \ge 3$, the digraph $(\dC_m \cartprod \dC_n) \smallsetminus R_{2,2}$ is hamiltonian if and only if: 
	\begin{enumerate}
	\item \label{WuThm2-gcd=1}
	$\gcd(m, n) = 1$, 
	\item\label{WuThm2-a}
	 $\min \{\, \crt{0}{-2}, \,\crt{-2}{0} \,\} < \min \bigl\{\,  \crt{0}{-1}, \, \crt{-1}{0} \, \bigr\}$,
	and
	\item \label{WuThm2-gcd}
	$\displaystyle \gcd \left( \frac{\crt0{-4}}{m} , \ \frac{\crt{-4}{0}}{n} \right) = 1$.
	\end{enumerate}
\end{cor}

\begin{proof}[\bf Proof of \cref{main}]
Combine \cref{PushedIffDeleted,WuThm2}.
\end{proof}

\section{Proof of \texorpdfstring{\cref{WuThm2} from \cref{WuThm}}{the corollary from the theorem}} \label{DerivePfSect}

We now explain how to derive \cref{WuThm2} from \cref{WuThm}. (Alternatively, the \lcnamecref{WuThm2} could also be derived from the work of S.\,J.\,Curran et al.\ \cite[Thm.~4.3]{CurranEtAl}, or see \cref{DirectPfSect} for a direct proof that does not assume familiarity with \cite{CurranEtAl} or~\cite{Wu}.) Actually, we prove only ($\Rightarrow$) in this \lcnamecref{DerivePfSect}, but the argument is reversible.

The conclusions of the \lcnamecref{WuThm2} are symmetric under interchanging $m$ and~$n$, so we may assume that the conditions in the statement of \cref{WuThm} hold. For $a = b = 2$, this means:
	\begin{align} \label{WuThm2Pf-a}
	 \crt{-2}{0} = \min \left\{ 
	 	\begin{matrix} \crt{-2}{-2}, \ \crt{-2}{-1}, \ \crt{-2}{0}, \\[\smallskipamount]
		 \crt{-2}{-2}, \ \crt{-1}{-2}, \ \crt{0}{-2} \end{matrix} \right\}
	\end{align}
and
	\begin{align} \label{WuThm2Pf-gcd}
	\gcd \left( n - 2 - 2 \left\lfloor \frac{\crt{-2}{0}}{m} \right\rfloor, \ 2 \, \frac{\crt{-2}{0}}{n} \right) = 1
	. \end{align}

\pref{WuThm2-gcd=1}
Note that $n$ must be odd. (Otherwise, both terms in the $\gcd$ of~\pref{WuThm2Pf-gcd} are even, which contradicts the fact that the $\gcd$ is~$1$.) Also, since $\crt{-2}{0}$ exists, we know that $\gcd(m,n) \in \{1, 2\}$. From the fact that $n$ is odd, we conclude that $\gcd(m,n) = 1$.

\pref{WuThm2-a}
Since 
	\begin{align} \label{SpanIffPf-diff}
	\text{$\crt{i - 1}{j - 1} = \crt{i}{j} - 1$ 
	\quad
	(unless $i \equiv j \equiv 0 \pmod{\lcm(m,n)}$),} 
	\end{align}
we have 
	\begin{align} \label{SpanIffPf-12=01}
	\text{$\crt{-2}{-1} = \crt{-1}{0} - 1$
	\quad and \quad
	$\crt{-1}{-2} = \crt{0}{-1} - 1$} 
	. \end{align}
Therefore, we see from~\pref{WuThm2Pf-a} that \pref{WuThm2-a} holds.

(For reversing the argument, note that $\crt{-2}{-2} = mn - 2$, so it is always true that $\crt{-2}{0} < \crt{-2}{-2}$, and also note that the inequality $\crt{-2}{0} < \crt{0}{-2}$ can be achieved by interchanging $m$ and~$n$ if it does not already hold.)

\pref{WuThm2-gcd}
Note that 
	\[ \crt{-2}{0} + \crt{0}{-2} = mn - 2, \]
because the left-hand side is congruent to~$-2$ modulo both $m$ and~$n$ (and we know from~\pref{WuThm2-gcd=1} that $m$ and~$n$ are relatively prime). 
Since (by~\pref{WuThm2Pf-a}) we have $\crt{-2}{0} < \crt{0}{-2}$,
this implies that 
	\[ \crt{-2}{0} < \frac{mn}{2} - 1, \]
so $\crt{-4}{0} = 2 \, \crt{-2}{0}$. Therefore, we have
	\[ 2 \, \frac{\crt{-2}{0}}{n} = \frac{\crt{-4}{0}}{n} . \]

Hence, in order to establish that~\pref{WuThm2Pf-gcd} is the same as conclusion~\pref{WuThm2-gcd} of the \lcnamecref{WuThm2}, all that remains is to show
	\[ n - 2 - 2 \left\lfloor \frac{\crt{-2}{0}}{m} \right\rfloor = \frac{\crt0{-4}}m . \]
Since $\crt{-2}{0} + 2$ is a multiple of~$m$, we see that the left-hand side is
	\begin{align*}
	n - 2 - 2 \left( \frac{\crt{-2}{0} + 2}{m} - 1 \right)
	&= n - 2 \, \frac{\crt{-2}{0} + 2}{m} 
	\\&= \frac{mn - 2 \, \crt{-2}{0} - 4}{m} 
	. \end{align*}
Also note that
	\[ mn - 2 \, \crt{-2}{0} - 4
	> mn - 2 \left( \frac{mn}{2} - 1 \right) - 4
	= -2 . \]
It is therefore easy to see that 
	\[ mn - 2 \, \crt{-2}{0} - 4 = \crt{0}{-4} \]
(because the two sides are congruent modulo both~$m$ and~$n$), which completes the proof.

\section{Direct proof of \texorpdfstring{\cref{WuThm2}}{the corollary}} \label{DirectPfSect}

For completeness (since \cite{Wu} was never published), and because some readers may find it instructive, we sketch a direct proof of \cref{WuThm2} that is based on S.\,X.\,Wu's proof \cite[\S4]{Wu} of \cref{WuThm}. (The same ideas apply to the general case of \cref{WuThm}, but the details are more complicated.) We begin with two % !!!
\lcnamecref{TravelDefn}s and some \lcnamecref{SpanIff}s.

\begin{defn}[\hskip-0.001pt{\cite{Wikipedia-CycleCover}}] % not for publication version !!!
A spanning subdigraph~$H$ of a digraph~$X$ is a \emph{vertex-disjoint cycle cover} if $H$ is a vertex-disjoint union of directed cycles. (Equivalently, the invalence and outvalence of every vertex of~$H$ is~$1$.)
\end{defn}

\begin{defn}[\hskip-0.001pt{cf.\ \cite[p.~82]{Housman}}] \label{TravelDefn} % not for publication version !!!
Assume $H$ is a vertex-disjoint cycle cover of $(\dC_m \cartprod \dC_n) \smallsetminus R_{2,2}$. Let $v$ be a vertex of~$H$, and let $s \in \{(1,0), (0,1)\}$. We say that \emph{$v$ travels by~$s$} if the out-edge of~$v$ is the directed edge from~$v$ to $v + s$.
\end{defn}

The arguments in this \lcnamecref{DirectPfSect} utilize basic properties of the ``arc-forcing subgroup'' $\langle (1, -1) \rangle$ \cite[\S2.3]{WitteGallianSurvey} that were discovered by R.\,A.\,Rankin \cite[Lem.~1]{Rankin} and D.\,Housman \cite[pp.~82--83]{Housman}. The specific facts that we need are recorded in the following % !!!
\lcnamecref{ArcForcing}.

\begin{lem}[cf.~{\cite[p.~2]{Wu} or \cite[Rem.~2.2]{CurranEtAl}}] \label{ArcForcing}
Let $H$ be a vertex-disjoint cycle cover of $(\dC_m \cartprod \dC_n) \smallsetminus R_{2,2}$. For every vertex~$v$ of~$H$:
	\begin{enumerate}
	\item \label{ArcForcing-10}
	If $v$ travels by $(1,0)$, and $v + (1, -1) \notin R_{2,2}$, then $v + (1, -1)$ also travels by $(1,0)$.
	\item \label{ArcForcing-10x}
	If $v$ travels by $(1,0)$, and neither $v - (1, -1)$ nor $v + (0,1)$ is in~$R_{2,2}$, then $v - (1, -1)$ also travels by $(1,0)$.
	\item \label{ArcForcing-01}
	If $v$ travels by $(0,1)$, and $v - (1, -1) \notin R_{2,2}$, then $v - (1, -1)$ also travels by $(0,1)$.
	\item \label{ArcForcing-01x}
	If $v$ travels by $(0,1)$, and neither $v + (1, -1)$ nor $v + (1,0)$ is in~$R_{2,2}$, then $v + (1, -1)$ also travels by $(0,1)$.
	\end{enumerate}
\end{lem}

\begin{proof}
(\ref{ArcForcing-10}, \ref{ArcForcing-01}) By symmetry, it suffices to prove~\pref{ArcForcing-10}. For convenience, let $w = v + (1,0)$. 
Since $H$ is a vertex-disjoint cycle cover, we know that $H$ cannot have both a directed edge from $v$ to~$w$ and a directed edge from $v + (1, -1)$ to~$w$ (because the invalence of~$w$ cannot be greater than~$1$). Hence, $v + (1, -1)$ cannot travel by $(0,1)$. However, $v + (1, -1)$ is a vertex of~$H$ (because, by assumption, it is not in $R_{2,2}$), so it must  have some out-edge. We conclude that it travels by $(1,0)$, since that is the only other possibility.

(\ref{ArcForcing-10x}, \ref{ArcForcing-01x}) By symmetry, it suffices to prove~\pref{ArcForcing-10x}. For convenience, let $w = v + (0,1)$. 
Since $v$ travels by $(1,0)$, it does not travel by $(0,1)$, so $H$ does not contain the directed edge from $v$ to~$w$. Since $w$ must have an in-edge, this implies that $H$ has the directed edge from $v - (1, -1)$ to~$w$ (since that is the only other possibility). This means that $v - (1, -1)$ travels by $(1,0)$.
\end{proof}

\begin{lem}[Wu {\cite[Lem.~1]{Wu} or \cite[Lem.~2.3]{CurranEtAl}}] \label{UniqueCycles}
For $m,n \ge 2$, the digraph $(\dC_m \cartprod \dC_n) \smallsetminus R_{2,2}$ has no more than one vertex-disjoint cycle cover.
\end{lem}

\begin{proof}
Let $H$ be a vertex-disjoint cycle cover.

We claim that every coset of the subgroup $\langle (1, -1) \rangle$ contains at least one element of $R_{2,2}$. Suppose not, so we may let $v + \langle (1,-1) \rangle$ be a coset that does not intersect the set~$R_{2,2}$. By symmetry, we may assume, without loss of generality, that $v$ travels by $(1,0)$. Then, by repeated application of \fullcref{ArcForcing}{10}, we conclude that every element of this coset travels by $(1,0)$. Since the terminal endpoint of every directed edge of~$H$ must be a vertex of~$H$, this implies that every element of the coset $v + (1,0) + \langle (1,-1) \rangle$ is a vertex of~$H$. In other words, this coset does not intersect the set~$R_{2,2}$. By repeating this argument, we conclude, for every $k \in \ZZ^+$, that the coset $v + (k,0) + \langle (1,-1) \rangle$ does not intersect~$R_{2,2}$. However, the union of these cosets is all of $\dC_m \cartprod \dC_n$. We conclude that $R_{2,2}$ has no elements, which is a contradiction.

The claim implies that every vertex of~$H$ is contained in set of the form
	\[ I_{v,k} = \{ v, \ v + (1, -1), \ v + 2(1, -1), \ \ldots, \ v + k (1, -1) \} , \]
such that 
	\begin{enumerate} 
	\renewcommand{\theenumi}{\alph{enumi}}
	\item \label{UniqueCyclesPf-endpt}
	either $v - (1, -1) \in R_{2,2}$ or $v + (0,1) \in R_{2,2}$,
	\item either $v + (k + 1) \, (1, -1) \in R_{2,2}$ or $v + (k + 1) + (1,0) \in R_{2,2}$,
	\end{enumerate}
but
	\begin{enumerate} 
	\renewcommand{\theenumi}{\alph{enumi}}
	\setcounter{enumi}{2}
	\item \label{UniqueCyclesPf-notI}
	no element of~$I_{v,k}$ is in~$R_{2,2}$, and
	\item \label{UniqueCyclesPf-notNeigh}
	for $1 \le j < k$, neither $v + j(1, -1) + (1,0)$ nor $v + j(1, -1) + (1,0)$ is in~$R_{2,2}$.
	\end{enumerate}
From \pref{UniqueCyclesPf-notI} and~\pref{UniqueCyclesPf-notNeigh} (combined with \cref{ArcForcing}) and induction, we see that either every element of $I_{v,k}$ travels by $(1,0)$, or every element of~$I_{v,k}$ travels by $(0,1)$. 

To complete the proof, we will show that there is no choice about whether these vertices travel by $(1,0)$ or by $(1,0)$: it is uniquely determined for each~$v$. First of all, if $v + (0,1) \in R_{2,2}$, then $v$ cannot travel by~$(0,1)$, so it must travel by~$(1,0)$; hence, every vertex in $I_{v,k}$ must travel by $(0,1)$. On the other hand, if $v + (0,1) \notin R_{2,2}$, then (by~\pref{UniqueCyclesPf-endpt}) we must have $v - (1, -1) \in R_{2,2}$, so the vertex $v - (1, -1)$ is not in~$H$, and therefore cannot travel by $(1,0)$. Since $v + (0,1)$ must have an in-edge, we conclude that $v$ travels by $(0,1)$; hence, every vertex in $I_{v,k}$ must travel by $(0,1)$.
\end{proof}

\begin{lem}[cf.\ {\cite[Thm.~10]{Wu}}] \label{SpanIff}
For $m,n \ge 3$, the digraph $(\dC_m \cartprod \dC_n) \smallsetminus R_{2,2}$ has a vertex-disjoint cycle cover if and only if $\crt{-2}{0}$ and $\crt{0}{-2}$ exist, and
	\[ \min \bigl\{\, \crt{-2}{0} , \ \crt{0}{-2}  \,\bigr\} 
		< \min \bigl\{\, \crt{0}{-1} , \ \crt{-1}{0} \,\bigr\} 
		 . \]
Furthermore, if the digraph does have a vertex-disjoint cycle cover, then the number of vertices that travel by $(1,0)$ in this subdigraph is exactly twice the left-hand side of the above % !!!
 inequality.
\end{lem}

\begin{proof}
($\Rightarrow$)
If a vertex~$u$ travels by $(1,0)$, and we let 
	\begin{align} \label{SpanIffPf-rDefn}
	r(u) = \min \{\, k \in \ZZ^+ \mid u + k (1, -1) \in R_{2,2} \,\}
	, \end{align}
then it follows from \fullcref{ArcForcing}{10} (and induction on~$k$) that $u + k (1,-1)$ travels by $(1,0)$ for $0 \le k < r(u)$. This implies
	\[ \text{$u + k (1,-1) + (1,0) \notin R_{2,2}$ for $0 \le k < r(u)$} . \]

In particular, we can apply this with $u = (1, -1)$, since this vertex travels by $(1,0)$ because $u + (0,1) = (1,0) \in R_{2,2}$. We have 
	\begin{align}
	R_{2,2} 
	&= \left\{ \begin{matrix} (0,1), & (1,1), \\ (0,0), & (1,0) \phantom{,} \end{matrix} \right\}
		\notag
	\\&= \left\{ \begin{matrix} (1, -1) + (-1,2), & (1, -1) + (0,2), \\ (1, -1) + (-1, 1) , & (1, -1) + (0,1)\phantom{,} \end{matrix} \right\} 
		\label{SpanIffPf-R22}
	\\&=  \left\{ \begin{matrix}  u + \crt{-1}{-2} \cdot (1, -1), &u + \crt{0}{-2} \cdot (1, -1), \\
	u + \crt{-1}{-1} \cdot (1, -1), & u + \crt{0}{-1} \cdot (1, -1) \phantom{,} \end{matrix} \right\} 
		\notag
	, \end{align}
so
	\begin{align} \label{SpanIffPf-r(u)}
	r(u) 
	%= \min \{\, k \in \ZZ^+ \mid u + k (1, -1) \in R_{2,2} \,\}
		=  \min\{\, \crt{-1}{-2}, \ \crt{0}{-2}, \ \crt{-1}{-1}, \ \crt{0}{-1} \} 
	. \end{align}
Also, since 
	\begin{align*}
	&u + \crt{-2}{-1} \cdot (1, -1) + (1,0) = (1, -1) + (-2,1) + (1,0) = (0, 0) \in R_{2,2}
\intertext{and}
	& u + \crt{-2}{-2} \cdot (1, -1) + (1,0) = (1,-1) + (-2, 2) + (1,0) = (0, 1) \in R_{2,2}
	, \end{align*}
we know that $u + \crt{-2}{-1} \cdot (1, -1)$ and $u + \crt{-2}{-2} \cdot (1, -1)$ do not travel by $(1,0)$, so
	\begin{align} \label{SpanIffPf-ineq}
	r(u) \le \min \bigl\{\, \crt{-2}{-1}, \ \crt{-2}{-2} \,\}
	. \end{align}
Since $\crt{-2}{-2} = \lcm(m,n) - 2$ is very large, it is almost entirely irrelevant in~\pref{SpanIffPf-ineq}, but it does imply that $\crt{-1}{-1} = \lcm(m,n) - 1$ is not the only term that exists in the right-hand side of~\pref{SpanIffPf-r(u)}. This implies that 
	$\gcd(m,n) \in \{1,2\}$, so
	$\crt{-2}{0}$ and $\crt{0}{-2}$ exist.

We may now assume that $\crt{0}{-1}$ and (equivalently) $\crt{-1}{0}$ exist, for otherwise the inequality in the statement of the \lcnamecref{SpanIff} is vacuously true.
We may also assume (by interchanging $m$ and~$n$ if necessary) that 
	\[ \min \bigl\{\, \crt{0}{-1} , \ \crt{-1}{0} \,\bigr\}  = \crt{-1}{0} . \]
Thus, we see from~\pref{SpanIffPf-12=01} that~\pref{SpanIffPf-ineq} is equivalent to the condition that $\crt{0}{-2} < \crt{-1}{0}$. This establishes the inequality in the statement of the \lcnamecref{SpanIff}.

($\Leftarrow$)
Let $u = (1,-1)$ and assume, without loss of generality, that 
	\[ \crt{-1}{0} < \crt{0}{-1} . \]
Since 
	\[ \crt{-1}{0} + \crt{0}{-1} = \crt{-1}{-1} = \lcm(m,n) - 1 , \]
this implies that $\crt{-1}{0} < \lcm(m,n)/2$, so $\crt{-2}{0} = 2 \, \crt{-1}{0} > \crt{-1}{0}$. So we see from the assumption of this direction of the proof that
	\begin{align} \label{SpanIffPf-02<10}
	\crt{0}{-2} < \crt{-1}{0} = \min \bigl\{\, \crt{-1}{0} , \  \crt{0}{-1} , \  \crt{-2}{0} \,\bigr\} 
	. \end{align}
We then conclude from~\pref{SpanIffPf-R22} (and the definition of $r(u)$ in~\pref{SpanIffPf-rDefn}) that
	\begin{align} \label{SpanIffPf-r=02}
	r(u) = \crt{0}{-2} 
	\end{align}
and (using~\pref{SpanIffPf-diff}) that
	\[ \text{$u + k(1,-1) + (1,0) \notin R_{2,2}$ \ for $0 \le k < r(u)$.} \]

Let $u' = u - (1,0) = (0, -1)$. We claim that 
	\begin{align} \label{r=r}
	 r(u') = r(u) 
	 . \end{align}
To see this, first note that
	\[ u' + r(u) \cdot (1, -1) 
	= (0, -1) +  \crt{0}{-2} \cdot (1, -1) 
	= (0,-1) + (0,2)
	= (0,1) 
	\in R_{2,2} , \]
so $r(u') \le r(u)$. On the other hand, we have
	\begin{align*}
	R_{2,2} 
	&= \left\{ \begin{matrix} (0,1), & (1,1), \\ (0,0), & (1,0) \phantom{,} \end{matrix} \right\}
	\\&= \left\{ \begin{matrix} (0,-1) + (0,2), & (0,-1) + (1,2), \\ (0,-1) + (0, 1) , & (0,-1) + (1,1)\phantom{,} \end{matrix} \right\} 
	\\&=  \left\{ \begin{matrix}  u' + \crt{0}{-2} \cdot (1, -1), &u' + \crt{1}{-2} \cdot (1, -1), \\
	u' + \crt{0}{-1} \cdot (1, -1), & u' + \crt{1}{-1} \cdot (1, -1) \phantom{,} \end{matrix} \right\} 
	, \end{align*}
so
	\begin{align*}
	r(u') 
	&= \min \bigl\{\, \crt{0}{-2}, \ \crt{1}{-2},\phantom{{}+1} \ \crt{0}{-1}, \ \crt{1}{-1}\phantom{{}+1}  \,\bigr\}
	\\[\smallskipamount]&= \min \bigl\{\, \crt{0}{-2}, \ \crt{0}{-3} + 1, \ \crt{0}{-1}, \ \crt{0}{-2} + 1 \,\bigr\}
	\end{align*}
From~\pref{SpanIffPf-02<10}, we know that the only value in this minimum that could possibly be smaller than $r(u) = \crt{0}{-2}$ is $\crt{0}{-3} + 1$. However, we have
	\[ \crt{0}{-2} + \crt{0}{-1} < \crt{-1}{0} + \crt{0}{-1} = \lcm(m,n) - 1 < \lcm(m,n) , \]
so 
	\[ \crt{0}{-3} = \crt{0}{-2} + \crt{0}{-1} > \crt{0}{-2} = r(u) . \]
This completes the proof of the claim.

Note that, for $0 \le k < r(u)$, we have
	\[ u' + k(1, -1) + (1,0) = u + k(1, -1) \notin R_{2,2} . \]
Also note that $u$ and~$u'$ are the only vertices in $(\dC_m \cartprod \dC_n) \smallsetminus R_{2,2}$ that cannot travel by $(0,1)$. Therefore, we can construct a spanning subdigraph~$H$ in which a vertex travels by $(1,0)$ if it is in the set 
	\[ \left\{ v + k(1, -1) \;\middle|\; \begin{matrix} v \in \{u,u'\}, \\ 0 \le k < r(u) \end{matrix} \right\}  \]
and travels by $(0,1)$ otherwise. By construction (and~\pref{r=r}), if a vertex~$v$ travels by $(1,0)$, and $v + (1,-1) \notin R_{2,2}$, then $v + (1,-1)$ also travels by $(1,0)$. Hence, no vertex has invalence~$2$, so the in-degree (and out-degree) of every vertex of~$H$ is~$1$, which means that $H$ is the desired vertex-disjoint cycle cover.

Furthermore, we know from the construction of~$H$ (together with~\pref{SpanIffPf-02<10} and~\pref{SpanIffPf-r=02}) that the number of vertices that travel by $(1,0)$ is as specified in the final sentence of the statement of the \lcnamecref{SpanIff}. Since $H$ is the only vertex-disjoint cycle cover (see \cref{UniqueCycles}), this completes the proof.
 \end{proof}
 
 \begin{proof}[\bf Direct proof of \cref{WuThm2}]
 \Cref{SpanIff} provides necessary conditions for the existence of a hamiltonian cycle (in particular, \fullref{main}{20} must hold), but they is not sufficient, because we need an additional condition that determines whether the cycle cover is a single cycle, rather than a union of several cycles.  This condition is provided by the ``knot class\rlap,'' which is a topological concept that was introduced into the study of Cartesian products of directed cycles by S.\,J.\,Curran \cite[\S4]{CurranWitte}.
 
 Namely, let $H$ be the vertex-disjoint cycle cover, and suppose the number of vertices of~$H$ that travel by $(1,0)$ is~$x$, and the number that travel by~$(0,1)$ is~$y$. Then $x/m$ and $y/n$ are integers, and the knot class of~$H$ is defined to be the ordered pair $(x/m, y/n)$ \cite[Rem.~4.5]{CurranWitte}. The theory \cite[Prop.~4.12(a)]{CurranWitte} tells us that 
 	\begin{align} \label{CycleIffGcd}
	\text{$H$ consists of a single cycle if and only if $\gcd(x/m, y/n) = 1$.}
	\end{align}

 We now use~\pref{CycleIffGcd} to show that \fullref{main}{mn} is a necessary condition for $H$ to be a hamiltonian cycle. The key is to notice that if $\gcd(m,n) \neq 1$, then since \cref{SpanIff} tells us that $\crt{-2}{0}$ exists, we must have $\gcd(m,n) = 2$, so 
 	\[ \text{$m$ and~$n$ are even.} \]
However, if we assume, without loss of generality, that $\crt{0}{-2} < \crt{-2}{0}$, then the last sentence of \cref{SpanIff} tells us that
	\[ x = 2 \, \crt{0}{-2} . \]
Since the number of vertices of~$H$ is $mn - 4$, this implies
	\[ y = mn - 4 - x = mn - 4 - 2 \, \crt{0}{-2} = mn - 2 \, \crt{2}{0} . \]
Since $m$ is even, it is now obvious that
	\[ \frac{x}{m} = 2 \, \frac{\crt{0}{-2}}{m}
	\quad \text{and} \quad
	\frac{y}{n} = m - 2 \, \frac{\crt{2}{0}}{n} \]
are even. Hence, $\gcd(x/m, y/n) \neq 1$, so we see from~\pref{CycleIffGcd} $H$ is not a hamiltonian cycle.

To complete the proof, we now consider the situation where \fullref{main}{mn} holds, which means that $\gcd(m,n) = 1$.
Since $x/m$ and $y/n$ are integers, we know that 
	\[ \text{$x \equiv 0 \pmod{m}$ \ and \ $y \equiv 0 \pmod{n}$.} \]
Also, since the number of vertices of~$H$ is $mn - 4$, we know that 
	\[ x + y = mn - 4 \equiv - 4 \pmod{mn} . \]
Combining these congruences tells us that 
	\[ \text{$x \equiv -4 \pmod{n}$ \ and \ $y \equiv -4 \pmod{m}$.} \]
So $x = \crt{0}{-4}$ and $y = \crt{-4}{0}$. We conclude from~\pref{CycleIffGcd} that $H$~consists of a single cycle (and is therefore a hamiltonian cycle) if and only if the condition in \fullref{main}{40} holds.
 \end{proof}


\begin{thebibliography}{9}

\bibitem{CurranEtAl}
S.\,J.\,Curran, M.\,N.\,Ferencak, C.\,J.\,Morgan, and J.\,W.\,Thompson:
The hamiltonicity of a Cayley digraph of a modified abelian group on two generators.
\emph{Congr. Numerantium} 188 (2007) 75--95.
\Zbl{1136.05034}, 
\MR{2408756}

\bibitem{CurranWitte}
S.\,J.\,Curran and D.\,Witte:
Hamilton paths in Cartesian products of directed cycles. 
In \emph{Cycles in Graphs}, edited by B.\,R.\,Alspach and C.\,D.\,Godsil. 
North-Holland Publishing Co., Amsterdam, 1985, pp.~35--74. 
\Zbl{0588.05025},
\MR{0821505},
\doi{10.1016/S0304-0208(08)72996-7}
% \url{https://www.elsevier.com/books/cycles-in-graphs/alspach/978-0-444-87803-8}

\bibitem{GallianBook}
J.\,A.\,Gallian:
\emph{Contemporary Abstract Algebra}, 10th ed. CRC Press, Boca Raton, FL, 2021.
\Zbl{1452.00001},
\doi{10.1201/9781003142331}

\bibitem{ProductHandbook}
R.\,Hammack, W.\,Imrich, and S.\,Klav\v{z}ar:
\emph{Handbook of Product Graphs}, 2nd ed.
CRC Press, Boca Raton, FL, 2011.
\Zbl{1283.05001}, \MR{2817074},
\doi{10.1201/b10959}
%\url{https://www.routledge.com/9781138199088}

%\bibitem{HolsztynskiStrube}
%W.\,Holszty\'nski and R.\,F.\,E.\,Strube:
%Paths and circuits in finite groups.
%\emph{Discrete Math.} 22 (1978) 263--272. 
%\Zbl{0384.20022},
%\MR{0522721},
%\doi{10.1016/0012-365X(78)90059-6}

\bibitem{Housman}
S.\,Housman:
Enumeration of Hamiltonian paths in Cayley diagrams.
\emph{Aequationes Math.} 23 (1981) 80--97. 
\Zbl{0494.05029}
\MR{0667220},
\doi{10.1007/BF02188014}

\bibitem{KlerleinCarr}
J.\,B.\,Klerlein and E.\,C.\,Carr:
Hamiltonicity in $C_n \times C_m$ after a single push.
\emph{Congr. Numerantium} 190 (2008) 87--96.
\Zbl{1165.05017},
\MR{2489794}

\bibitem{Klostermeyer}
W.\,F.\,Klostermeyer:
Pushing vertices and orienting edges.
\emph{Ars Combin.} 51 (1999) 65--75. 
\Zbl{0977.05057}, 
\MR{1675104}

\bibitem{Rankin}
R.\,A.\,Rankin:
A campanological problem in group theory. 
\emph{Proc. Cambridge Philos. Soc.} 44 (1948) 17--25.
\Zbl{0030.10606},
 \MR{0022846},
 \doi{10.1017/s030500410002394x}
 
 \bibitem{Wikipedia-CycleCover}
 Wikipedia, Vertex cycle cover.
 \url{https://en.wikipedia.org/wiki/Vertex_cycle_cover}

\bibitem{WitteGallianSurvey}
D.\,Witte and J.\,A.\,Gallian:
A survey: Hamiltonian cycles in Cayley graphs.
\emph{Discrete Math.} 51 (1984) 293--304. 
\Zbl{0712.05039},
\MR{0762322},
\doi{10.1016/0012-365X(84)90010-4}

\bibitem{Wu}
S.\,X.\,Wu: 
Cycles in the Cartesian product of two directed cycles, unpublished, 2006.

\end{thebibliography}
\end{document}